\documentclass[10pt]{amsart}
\usepackage{amssymb}
\usepackage{bm}
\usepackage{graphicx}
\usepackage[centertags]{amsmath}
\usepackage{amsfonts}
\usepackage{amsthm}
\newtheorem{thm}{Theorem}
\newtheorem{cor}[thm]{Corollary}
\newtheorem{lem}[thm]{Lemma}

\newtheorem{prop}[thm]{Proposition}
\newtheorem{fact}[thm]{Fact}

\newtheorem{problem}{Problem}
\theoremstyle{definition}
\newtheorem{rem}{Remark}

\newcommand{\rr}{\mathbb{R}}
\newcommand{\nn}{\mathbb{N}}
\newcommand{\con}{\smallfrown}
\newcommand{\ee}{\varepsilon}

\newcommand{\ct}{2^{<\nn}}

\newcommand{\spw}{\mathrm{SP_w}}

\newcommand{\lex}{<_{\mathrm{lex}}}
\newcommand{\spc}{\mathrm{SPC}}
\newcommand{\sg}{\sigma}

\begin{document}

\title{On antichains of spreading models of Banach spaces}
\author{Pandelis Dodos}
\address{National Technical University of Athens, Faculty of Applied
Sciences, Department of Mathematics, Zografou Campus, 157 80,
Athens, Greece} \email{pdodos@math.ntua.gr}
\maketitle

\footnotetext[1]{2000 \textit{Mathematics Subject Classification}: 46B20, 03E15.}


\begin{abstract}
We show that for every separable Banach space $X$,
either $\spw(X)$ (the set of all spreading models
of $X$ generated by weakly-null sequences in $X$, modulo
equivalence) is countable, or $\spw(X)$ contains an
antichain of the size of the continuum. This answers
a question of S. J. Dilworth, E. Odell and B. Sari.
\end{abstract}


\section{Introduction}

Let $X$ be a separable Banach space and denote by $\spw(X)$
the set of all spreading models of $X$ generated by weakly-null
sequences in $X$, modulo equivalence. By $\leq$ we denote the
usual relation on $\spw(X)$ of domination. The study of the
structure $(\spw(X),\leq)$ has been initiated by G. Androulakis,
E. Odell, Th. Schlumprecht and N. Tomczak-Jaegermann in \cite{AOST}.
They showed, for instance, that $(\spw(X), \leq)$ is a semi-lattice,
i.e. any two elements of $\spw(X)$ admit a least upper bound.
The question of determining which countable lattices can be realized
as $(\spw(X),\leq)$, for some separable Banach space $X$, has been
answered by S. J. Dilworth, E. Odell and B. Sari \cite{DOS}.

This note is motivated by the following problem posed by the
authors of \cite{DOS} (see \cite[Problem 1.13]{DOS}).
\begin{problem}
\label{pr1} If $\spw(X)$ is uncountable must there exist
$\{(x_n^\xi)_n:\xi<\omega_1\}$ in $\spw(X)$ which is either
strictly increasing with respect to $\xi$, or strictly decreasing,
or consists of mutually incomparable elements?
\end{problem}

To state our first result, let us say that a seminormalized
Schauder basic sequence $(x_n)_n$ in a Banach space $X$ is
$C$-\textit{Schreier spreading} for some $C\geq 1$ (or simply
Schreier spreading, if $C$ is understood) if for every
$k\in\nn$ and every $k\leq n_0<...<n_k$ and
$k\leq m_0<...<m_k$ we have that $(x_{n_i})_{i=0}^k$
is $C$-equivalent to $(x_{m_i})_{i=0}^k$. Observe that
if $(x_n)_n$ is Schreier spreading, then there exists a
unique spreading model (up to equivalence) generated
by subsequences of $(x_n)_n$. Denote by $\ct$ the Cantor tree
and let $\varphi:\ct\to\nn$ be the unique bijection
satisfying $\varphi(s)<\varphi(t)$ if either $|s|<|t|$, or $|s|=|t|=n$
and $s<_{\mathrm{lex}}t$ (here $<_{\mathrm{lex}}$ stands
for the usual lexicographical order on $2^n$). We show the following.
\begin{thm}
\label{t1} Let $X$ be a separable Banach space such that $\spw(X)$
is uncountable. Then there exist a family $(x_t)_{t\in\ct}$
in $X$ and $C\geq 1$ such that the following hold.
\begin{enumerate}
\item[(1)] If $(t_n)_n$ is the enumeration of $\ct$ according
to $\varphi$, then the sequence $(x_{t_n})_n$ is a seminormalized
Schauder basic sequence.
\item[(2)] For every $\sg\in 2^\nn$, the sequence $(x_{\sg|n})_n$
is weakly-null and $C$-Schreier spreading.
\item[(3)] For every $\sg, \tau\in 2^\nn$ with $\sg\neq \tau$,
if $(y^\sg_n)_n$ and $(y^\tau_n)_n$ are spreading models
generated by subsequences of $(x_{\sg|n})_n$ and $(x_{\tau|n})_n$
respectively, then $(y^\sg_n)_n$ and $(y^\tau_n)_n$ are
incomparable with respect to domination.
\end{enumerate}
\end{thm}
Theorem \ref{t1} implies the following.
\begin{cor}
\label{ic1} Let $X$ be a separable Banach space such that $\spw(X)$
is uncountable. Then $\spw(X)$ contains an antichain of the size of
the continuum.
\end{cor}
We notice that, independently, V. Ferenczi and C. Rosendal
have proved Corollary \ref{ic1} under the additional assumption
that $X$ has separable dual (\cite{FR}).

In \cite{AOST} (see also \cite{DOS}), it was shown that $\spw(X)$
can contain a strictly decreasing infinite sequence, yet no strictly
increasing infinite sequence can be found in $\spw(X)$. This is not,
however, the case of the uncountable.
\begin{thm}
\label{t2} Let $X$ be a separable Banach space.
\begin{enumerate}
\item[(a)] If $\spw(X)$ contains a strictly decreasing sequence
of length $\omega_1$, then $\spw(X)$ contains a strictly increasing
sequence of length $\omega_1$.
\end{enumerate}
On the other hand,
\begin{enumerate}
\item[(b)] if $\spw(X)$ does not contain a strictly increasing infinite
sequence, then there exists a countable ordinal $\xi_X$ such that
$\spw(X)$ does not contain strictly decreasing sequences of order
type greater than $\xi_X$.
\end{enumerate}
\end{thm}
It was shown in \cite[Theorem 3.7]{DOS} that for every countable
ordinal $\xi$ there exists a separable Banach space $X_\xi$ such
that $(\spw(X_\xi),\leq)$ does not contain a strictly increasing
infinite sequence, yet $\spw(X_\xi)$ contains a strictly decreasing
sequence of order type $\xi$. Thus, the ordinal $\xi_X$ obtained by
part (b) of Theorem \ref{t2} is not uniformly bounded within
the class of separable Banach spaces for which $\spw(X)$
does not contain a strictly increasing infinite sequence.

In the proofs of Theorem \ref{t1} and Theorem \ref{t2}(a) we
use the structural result obtained by B. Sari in \cite{S}.
The central argument, however, in the proof of Theorem \ref{t1} is
essentially based on the work of Leo Harrington and Saharon Shelah
on Borel orders. Deep as it is, the theory developed by
Harrington and Shelah is highly sophisticated. In particular,
all known proofs of their results use either Effective
Descriptive Set Theory or Forcing. However, for the proof
of Theorem \ref{t1} we need only some instances of
the theory and merely for $F_\sigma$ orders. Thus, we have
included ``elementary" proofs of all the results that
we need, making the paper essentially self-contained
and accessible to anyone with basic knowledge of
Classical Descriptive Set Theory. None of these proofs
should be considered as a contribution to the field of
Borel orders.

The paper is organized as follows. In \S 2 we state
and prove all results on Borel orders that are needed
for the proof of Theorem \ref{t1}. In \S 3 we show that
for every separable Banach space $X$ the structure
$(\spw(X),\leq)$ can be realized as an $F_\sigma$ order.
In \S 4 we give the proof of Theorem \ref{t1} while
the proof of Theorem \ref{t2} is given in \S 5.
\medskip

\noindent \textbf{Notations.} By $\nn=\{0,1,2,...\}$ we
denote the natural numbers while by $[\nn]$ the set of all
infinite subsets of $\nn$ (which is clearly a Polish subspace
of $2^\nn$). By $\ct$ we denote the set of all finite
sequences of 0's and 1's (the empty sequence is included).
We view $\ct$ as a tree equipped with the (strict) partial
order $\sqsubset$ of extension. For every $t\in\ct$ by
$|t|$ we denote the length of $t$, i.e. the cardinality
of the set $\{s\in \ct: s\sqsubset t\}$. For every $n\in\nn$
we let $2^n=\{t\in\ct: |t|=n\}$. If $s, t\in\ct$, then by
$s^{\con}t$ we denote their concatenation. For every
$\sg\in 2^\nn$ and every $n\geq 1$ we let $\sg|n=
\big(\sg(0),...,\sg(n-1)\big)$, while $\sg|0=(\varnothing)$.

If $(x_n)_n$ and $(y_n)_n$ are Schauder basic sequences
in a Banach space $X$ and $C\geq 1$, then we say that
$(x_n)_n$ is $C$-dominated by $(y_n)_n$ (or simply dominated,
if $C$ is understood) if for every $k\in\nn$ and every
$a_0,..., a_k\in\rr$ we have
\[ \Big\| \sum_{n=0}^k a_n x_n\Big\| \leq C
\Big\| \sum_{n=0}^k a_n y_n \Big\|.\]
We write $(x_n)_n \leq (y_n)_n$ to denote
the fact that $(x_n)_n$ is dominated by $(y_n)_n$.
All the other pieces of notation we use are standard
as can be found, for instance, in \cite{Kechris},
\cite{LT} or \cite{AOST}.
\medskip

\noindent \textbf{Acknowledgments.} I would like
to thank Spiros A. Argyros for many discussions
on the subject as well as for his comments on
the paper.


\section{Quasi-orders and Borel orders}
A \textit{quasi-order} is a set $X$ with a binary relation $\leq$ on
$X$ which is reflexive and transitive. For $x,y\in X$ we let
\begin{eqnarray*}
\mathrm{(a)} \ \ x\equiv y & \Leftrightarrow &(x\leq y)
\text{ and } (y\leq x) \\
\mathrm{(b)} \ \ x < y & \Leftrightarrow & (x \leq y )
\text{ and } (y\nleq x) \\
\mathrm{(c)} \ \ x \perp y &\Leftrightarrow & (x\nleq y)
\text{ and } (y\nleq x)
\end{eqnarray*}
If $x, y\in X$ are as in case (c) above, then we say that
$x$ and $y$ are \textit{incomparable}. An \textit{antichain} is a
subset of $X$ consisting of pairwise incomparable elements. An
\textit{$\omega_1$-chain} in $X$ is a sequence $(x_\xi)_{\xi<\omega_1}$
in $X$ such that either $x_\xi<x_\zeta$ for all $\xi<\zeta<\omega_1$
or $x_\xi<x_\zeta$ for all $\zeta<\xi<\omega_1$.

A \textit{Borel order} is a quasi-order $(X,\leq)$ where $X$ is Polish and
$\leq$ is Borel in $X^2$. A Borel order is called \textit{thin} if $X$
does not contain a perfect set of pairwise incomparable elements.
We will need the following lemma concerning the structure of $F_\sigma$
thin orders.
\begin{lem}
\label{ln1} Let $X$ be a Polish space and $\leq$ an $F_\sigma$ thin order
on $X$. Then $(X,\leq)$ does not contain $\omega_1$-chains.
\end{lem}
Lemma \ref{ln1} is a very special case of a deep result
due to L. Harrington and S. Shelah (see \cite{HS} and \cite{HMS})
asserting that \textit{any} Borel thin order does not contain
$\omega_1$-chains. We notice that, prior to \cite{HS},
H. Friedman had shown (\cite{F}) that any Borel linear order
does not contain $\omega_1$-chains.
\begin{proof}[Proof of Lemma \ref{ln1}]
Let $(F_n)_n$ be an increasing sequence of closed subsets
of $X^2$ with $\leq =\bigcup_n F_n$. By symmetry, it is enough to
show that if $(X,\leq)$ contains a strictly increasing sequence
$(x_\xi)_{\xi<\omega_1}$, then there exists a perfect subset $P$ of $X$
such that $x\perp y$ for all $x, y\in P$ with $x\neq y$. Set
$\Gamma=\{x_\xi:\xi<\omega_1\}$. Refining if necessary, we
may assume that for every $\xi<\omega_1$ the point $x_\xi$
is a condensation point of $\Gamma$. Let $\rho$ be a compatible
complete metric for $X$. By recursion on the length
of sequences in $\ct$, we shall construct a family
$(U_t)_{t\in \ct}$ of open subsets of $X$ such that
the following are satisfied.
\begin{enumerate}
\item[(a)] For every $t\in\ct$ we have $\overline{U}_{t^{\con}0},
\overline{U}_{t^{\con}1}\subseteq U_t$
and $\overline{U}_{t^{\con}0}\cap \overline{U}_{t^{\con}1}=\varnothing$.
\item[(b)] For every $t\in \ct$ with $|t|\geq 1$ we have
$\rho-\mathrm{diam}(U_t)\leq \frac{1}{|t|}$.
\item[(c)] For every $n\geq 1$ and every $t,s\in 2^n$ with
$t\neq s$ we have $(U_t\times U_s)\cap F_n=\varnothing$
and $(U_s\times U_t)\cap F_n=\varnothing$.
\item[(d)] For every $t\in\ct$, $U_t\cap \Gamma\neq\varnothing$.
\end{enumerate}
Assuming that the construction has been carried out, we set
\[ P= \bigcup_{\sigma\in 2^\nn} \bigcap_{n\in\nn} U_{\sigma|n}.\]
By (a) and (b) above, we see that $P$ is a perfect subset
of $X$. Moreover, using (c), it is easy to check that $P$ is
in addition an antichain.

We proceed to the construction. For $n=0$, we set $U_{(\varnothing)}=X$.
Let $\xi<\zeta<\omega_1$. Then $x_{\xi}<x_{\zeta}$, and so,
$x_{\zeta} \nleq x_{\xi}$. In particular, $(x_{\zeta}, x_{\xi})\notin
F_1$. Hence, there exist $V^0, W^0$ open subsets of $X$ such that
$x_\zeta\in V^0$, $x_\xi\in W^0$ and $(V^0\times W^0)\cap F_1=\varnothing$.
Notice that both $V^0\cap \Gamma$ and $W^0\cap \Gamma$ are uncountable.
So, we may select $\eta<\theta<\omega_1$ such that $x_\eta\in V^0$
and $x_\theta\in W^0$. As $x_\theta \nleq x_\eta$, we find $V^1, W^1$
open subsets of $V^0$ and $W^0$ respectively such that
$x_\theta\in W^1$, $x_\eta\in V^1$ and $(W^1\times V^1)\cap F_1=\varnothing$.
Notice that conditions (c) and (d) above are satisfied for
$V^1$ and $W^1$ except, possibly, (a) and (b).
Thus, refining, we find $U_{(0)}$ and $U_{(1)}$ open
subsets of $V^1$ and $W^1$ respectively such that conditions (a)-(d)
are satisfied. For the general step we proceed similarly.
The proof is completed.
\end{proof}
For more information on the structure of Borel thin orders
we refer to the work of A. Louveau \cite{L}, and A. Louveau
and J. Saint Raymond \cite{LStR}. For applications of the
theory of Borel orders to Banach space Theory we refer to
the work of C. Rosendal \cite{Ros}.

We will also need the following special case of the theorem
of J. H. Silver \cite{Si} on the number of equivalence
classes of co-analytic equivalence relations. The proof given
below is an adaptation of Louveau's approach on Silver's
theorem (via the, so called, Gandy-Harrington topology --
see \cite{MK}) in an easier setting.
\begin{lem}
\label{ln2} Let $X$ be a Polish space and $\sim$ an $F_\sigma$
equivalence relation on $X$. Then, either the equivalence classes of
$\sim$ are countable, or there exists a Cantor set $P\subseteq X$
consisting of pairwise inequivalent elements.
\end{lem}
\begin{proof}
Let $\mathcal{B}=(U_n)_n$ be a countable basis of $X$.
For every closed subset $F$ of $X$ let
\[ D(F)=F\setminus \bigcup\{U_n\in\mathcal{B}:
\exists x\in F \text{ with } U_n\cap F\subseteq [x]\} \]
where $[x]=\{y\in X: x\sim y\}$. That is, $D(F)$ results by removing
from $F$ all basic relatively open subsets of $F$ which are contained
in a single equivalence class. Clearly $D(F)$ is closed and
$D(F)\subseteq F$. By transfinite recursion, we define a decreasing
sequence $(X_\xi)_{\xi<\omega_1}$ of closed subsets of $X$
as follows. We set $X_0=X$, $X_{\xi+1}= D(X_\xi)$ and
$X_{\lambda}= \bigcap_{\xi<\lambda} X_\xi$
if $\lambda$ is limit. There exists $\xi_0<\omega_1$
such that $X_{\xi_0}=X_{\xi_0+1}$.
\medskip

\noindent \textit{Case 1}. $X_{\xi_0}=\varnothing$. Notice that
for every $\xi<\xi_0$ the set $X_{\xi}\setminus X_{\xi+1}$ is
contained in at most countable many equivalence classes. As
$X_{\xi_0}=\varnothing$, we see that
\[ X=\bigcup_{\xi<\xi_0} X_{\xi}\setminus X_{\xi+1}.\]
Hence, this case implies that the equivalence classes of $\sim$
are countable.
\medskip

\noindent \textit{Case 2.} $X_{\xi_0}\neq\varnothing$. We set
$Y=X_{\xi_0}$ and $\sim'=\sim\cap Y^2$. Clearly $\sim'$ is $F_\sigma$
in $Y^2$. We claim that $\sim'$ is meager in $Y^2$. By the
Kuratowski-Ulam Theorem (see \cite[Theorem 8.41]{Kechris}),
it is enough to show that for every $x\in Y$ the set
$[x]'=\{y\in Y: x\sim' y\}=\{y\in Y: x\sim y\}$ is meager.
Notice that $[x]'$ is $F_\sigma$ in $Y$. So, if $[x]'$ was not
meager, then there would existed $U_n\in\mathcal{B}$
such that $U_n\cap Y\subseteq [x]'$. This implies that
$D(X_{\xi_0})\varsubsetneq X_{\xi_0}$, a contradiction.
Thus, $\sim'$ is meager in $Y^2$ as claimed. It follows
by a classical result of Mycielski (see \cite[Theorem 19.1]{Kechris})
that there exists a Cantor set $P\subseteq Y$ such that
$x\nsim' y$ for all $x, y\in P$ with $x\neq y$. This clearly
implies that $x\nsim y$ for all $x, y\in P$ with $x\neq y$.
The proof is completed.
\end{proof}


\section{Coding $(\spw(X),\leq)$ as an $F_\sigma$ order}

Let $X$ be a separable Banach space. Our aim is to show that
the quasi-order $(\spw(X),\leq)$ can be realized as
an $F_\sigma$ order. This is done in a rather standard and
natural way.

Let $U$ be the universal space of A. Pelczynski for
unconditional basic sequences (see \cite{P}). That is,
$U$ has an unconditional Schauder basis $(u_n)_n$
and for any other unconditional Schauder basic
sequence $(y_n)_n$ in some Banach space $Y$ there
exists $L=\{l_0<l_1<...\}\in [\nn]$ such that
$(y_n)_n$ is equivalent to $(u_{l_n})_n$. In what follows,
for every $L=\{l_0<l_1<...\}\in [\nn]$ by $(u_n)_{n\in L}$
we denote the subsequence $(u_{l_n})_n$ of
$(u_n)_n$ determined by $L$. Define $\leq$ in
$[\nn]\times [\nn]$ by
\[ L\leq M \Leftrightarrow (u_n)_{n\in L} \text{ is dominated by }
(u_n)_{n\in M}. \]
Clearly $\leq$ is a quasi-order. Let $\sim$ be the associated
equivalence relation (i.e. $L\sim M$ if and only if
$L\leq M$ and $M\leq L$) and let $<$ be the strict
part of $\leq$ (i.e. $L<M$ if and only if $L\leq M$
and $M\nleq L$). Notice that $L\sim M$ if and only if
the sequences $(u_n)_{n\in L}$ and $(u_n)_{n\in M}$
are equivalent as Schauder basic sequences. We have
the following easy fact whose proof is sketched for
completeness.
\begin{fact}
\label{f1} Both $\leq$ and $\sim$ are $F_\sigma$.
\end{fact}
\begin{proof}
It is enough to show that $\leq$ is $F_\sigma$.
For every $K\in\nn$ with $K\geq 1$ let $\leq_K$ be the relation
on $[\nn]\times[\nn]$ defined by
\[ L\leq_K M \Leftrightarrow (u_n)_{n\in L} \text{ is }
K\text{-dominated by } (u_n)_{n\in M}.\]
It is easy to see that $\leq_K$ is closed in $[\nn]\times[\nn]$.
As $\leq$ is the union of $\leq_K$ over all $K\geq 1$, the
result follows.
\end{proof}
Our coding of $(\spw(X),\leq)$ as an $F_\sigma$ order will
be done via the following lemma.
\begin{lem}
\label{l1} Let $X$ be a separable Banach space. Then
there exists $A_X\subseteq [\nn]$ analytic such that
the following are satisfied.
\begin{enumerate}
\item[(1)] For every $(y_n)_n\in \spw(X)$ there exists
$L\in A_X$ such that $(y_n)_n$ is equivalent
to $(u_n)_{n\in L}$.
\item[(2)] For every $L\in A_X$ there exists $(y_n)_n\in\spw(X)$
such that $(u_n)_{n\in L}$ is equivalent to $(y_n)_n$.
\end{enumerate}
\end{lem}
\begin{proof}
Recall that a sequence $(x_n)_n$ in $X$ is said to be
\textit{Cesaro summable} if
\[ \lim_{n\to\infty} \frac{x_0+ ...+ x_{n-1}}{n}=0.\]
Let $\spc$ be the subset of $X^\nn$ defined by
\begin{eqnarray*}
(x_n)_n\in \spc & \Leftrightarrow & (x_n)_n \text{ is seminormalized,
Schauder basic, Cesaro summable}  \\
& & \text{and } C\text{-Schreier spreading for some } C\geq 1.
\end{eqnarray*}
It is easy to check that $\spc$ is a Borel subset of $X^\nn$
(actually, it is $F_{\sigma\delta}$). Consider the
subset $A$ of $[\nn]$ defined by
\begin{eqnarray*}
L\in A & \Leftrightarrow & \text{if } L=\{l_0<l_1<...\}, \text{ then }
\exists (x_n)_n\in X^\nn \ \exists \theta\geq 1 \text{ with }
\Big[ (x_n)_n\in \spc \\
& & \text{ and } \big( \forall k \ \forall k\leq n_0<
...< n_k \text{ we have } (x_{n_i})_{i=0}^{k}
\stackrel{\theta}{\sim} (u_{l_i})_{i=0}^{k} \big)\Big].
\end{eqnarray*}
As $\spc$ is Borel in $X^\nn$, it is easy to see that
the set $A$ is analytic. Denote by $(e_n)_n$ the standard
basis of $\ell_1$. Let us isolate the following
property of the set $A$.
\begin{enumerate}
\item[(P)] If $L\in A$, then the sequence $(u_n)_{n\in L}$
is not equivalent to $(e_n)_n$. This follows from the fact
that every sequence $(x_n)_n$ belonging to $\spc$ is a Cesaro
summable Schauder basic sequence.
\end{enumerate}
The proof of the lemma will be finished once we show the following.
\medskip

\noindent \textsc{Claim 1.} \textit{Let $(y_n)_n\in \spw(X)$
which is not equivalent to $(e_n)_n$. Then there exists $L\in A$
such that $(y_n)_n$ is equivalent to $(u_n)_{n\in L}$.
Conversely, for every $L\in A$ there exists $(y_n)_n\in \spw(X)$
which is not equivalent to $(e_n)_n$ and such that
$(u_n)_{n\in L}$ is equivalent to $(y_n)_n$.}
\medskip

\noindent \textit{Proof of Claim 1.} Let $(y_n)_n\in \spw(X)$
not equivalent to $(e_n)_n$ and let $(x_n)_n$ be a seminormalized
weakly-null sequence in $X$ that generates it. By passing to a
subsequence, we may assume that $(x_n)_n$ is a seminormalized,
$C$-Schreier spreading (for some $C\geq 1$) Schauder basic sequence.
As $(y_n)_n$ is not equivalent to $(e_n)_n$, by a result of H. P. Rosenthal
we see that $(x_n)_n$ has a subsequence $(x_{n_k})_k$ which is
additionally Cesaro summable (see \cite[Theorem II.9.8]{AT}).
Hence $(x_{n_k})_k\in\spc$. As $(x_{n_k})_k$ still generates
$(y_n)_n$ as spreading model, we easily see that there
exists $L\in A$ such that $(u_n)_{n\in L}$ is equivalent
to $(y_n)_n$.

Conversely, let $L\in A$. We pick $(x_n)_n\in \spc$
witnessing that $L\in A$. By property (P) above, we have that
$(u_n)_{n\in L}$ is not equivalent to $(e_n)_n$. Now
we claim that $(x_n)_n$ is weakly-null. Assume not. Then
there exist $M=\{m_0<m_1<...\}\in [\nn]$, $x^*\in X^*$ and
$\ee>0$ such that $x^*(x_{m_n})>\ee$ for every $n\in\nn$ (notice
also that $m_n\geq n$). Let $K\geq 1$ be the basis
constant of $(x_n)_n$. Let also $C\geq 1$ be such that
$(x_n)_n$ is $C$-Schreier spreading. Observe that
for every $n\in\nn$ we have
\begin{eqnarray*}
\Big\| \frac{x_0+...+x_{2n-1}}{2n}\Big\| & \geq &
\frac{1}{2(K+1)} \Big\| \frac{x_n+...+x_{2n-1}}{n}\Big\| \\
& \geq & \frac{1}{2C(K+1)} \Big\|
\frac{x_{m_n}+...+x_{m_{2n-1}}}{n}\Big\|\geq \frac{\ee}{2C(K+1)}
\end{eqnarray*}
which implies that $(x_n)_n$ is not Cesaro summable,
a contradiction. Thus, $(x_n)_n$ is weakly-null.
Let $(y_n)_n$ be a spreading model generated by
a subsequence of $(x_n)_n$. Then $(y_n)_n\in\spw(X)$.
Invoking the definition of the set $A$ again, we see that
$(y_n)_n$ is equivalent to $(u_n)_{n\in L}$. This yields
additionally that $(y_n)_n$ is not equivalent to
$(e_n)_n$. The proof of the claim is completed.  \hfill $\lozenge$
\medskip

\noindent If $(e_n)_n\notin \spw(X)$, then we set $A_X=A$.
If $(e_n)_n\in\spw(X)$, then we set $A_X=A\cup \{ L\in [\nn]:
(u_n)_{n\in L}\sim (e_n)_n\}$. Clearly $A_X$ is analytic and,
by Claim 1, $A_X$ is as desired. The lemma is proved.
\end{proof}


\section{Proof of Theorem \ref{t1}}

Let $X$ be a separable Banach space such that $\spw(X)$ is uncountable.
Let $A_X$ be the analytic subset of $[\nn]$ obtained by
Lemma \ref{l1}. We fix $\Phi:\nn^\nn\to [\nn]$ continuous with
$\Phi(\nn^\nn)=A_X$. We define $\precsim$ on $\nn^\nn$ by
\[ \alpha \precsim \beta \Leftrightarrow \Phi(\alpha)\leq \Phi(\beta).\]
By Fact \ref{f1} and the continuity of $\Phi$, we see
that $\precsim$ is an $F_\sigma$ quasi-order on
the Baire space $\nn^\nn$.
\begin{lem}
\label{l2} Let $X$ be a separable Banach space such that
$\spw(X)$ is uncountable and consider the $F_\sigma$ quasi-order
$(\nn^\nn, \precsim)$. Then, either
\begin{enumerate}
\item[(a)] $(\nn^\nn, \precsim)$ is not thin, or
\item[(b)] $(\nn^\nn, \precsim)$ contains a strictly increasing
sequence of length $\omega_1$.
\end{enumerate}
\end{lem}
\begin{proof}
Let $\cong$ be the equivalence relation associated with
$\precsim$ (i.e. $\alpha\cong \beta$ if $\alpha\precsim \beta$
and $\beta\precsim\alpha$). Notice that
\[ \alpha \cong \beta \Leftrightarrow \Phi(\alpha)\sim \Phi(\beta)\]
for every $\alpha, \beta \in \nn^\nn$. Also observe that
$\cong$ is an $F_\sigma$ equivalence relation. As $\spw(X)$
is uncountable, we see that $\cong$ has uncountable many equivalence classes.
Thus, by Lemma \ref{ln2}, there exists a Cantor set $P\subseteq \nn^\nn$
such that $\alpha\ncong\beta$ for every $\alpha, \beta\in P$
with $\alpha\neq\beta$. Fix a homeomorphism $h:2^\nn\to P$.
Let $\lex$ be the (strict) lexicographical ordering on $2^\nn$.
For every $Q\subseteq 2^\nn$, denote by $[Q]^2$
the set of unordered pairs of elements of $Q$.
Consider the following subsets $\mathcal{I}$ and
$\mathcal{D}$ of $[2^\nn]^2$ defined by
\[ \{\sigma, \tau\}\in \mathcal{I} \Leftrightarrow
\text{ if } \sigma\lex \tau \text{ then }
h(\sigma)\precsim h(\tau),\]
\[ \{\sigma, \tau\}\in \mathcal{D} \Leftrightarrow
\text{ if } \sigma\lex \tau \text{ then }
h(\tau)\precsim h(\sigma).\]
It is easy to check that both $\mathcal{I}$ and $\mathcal{D}$
are Borel in $[2^\nn]^2$, in the sense that
the sets
\[ \mathcal{I}^*=\big\{ (\sigma,\tau)\in 2^\nn\times 2^\nn:
\{\sigma,\tau\}\in \mathcal{I}\big\} \text{ and }
\mathcal{D}^*=\big\{ (\sigma,\tau)\in 2^\nn\times 2^\nn:
\{\sigma,\tau\}\in \mathcal{D}\big\}\]
are both Borel subsets of $2^\nn\times 2^\nn$. By result
of F. Galvin (see \cite[Theorem 19.7]{Kechris}), there exists
$Q\subseteq 2^\nn$ perfect such that one of the
following cases occur.
\medskip

\noindent \textit{Case 1.} $[Q]^2\subseteq \mathcal{I}$. We
fix a sequence $(\sigma_n)_n$ in $Q$ which is increasing
with respect to $\lex$. Then $h(\sigma_n)\precsim h(\sigma_m)$
for all $n<m$. As $h(Q)\subseteq P$ and $P$ consists of
inequivalent elements with respect to $\cong$, we see that the sequence
$\big( h(\sigma_n)\big)_n$ is strictly increasing.
This yields that $(\spw(X), \leq)$ contains
a strictly increasing sequence. By a result of B. Sari
\cite{S}, we conclude that $\spw(X)$ must contain a
strictly increasing sequence of length $\omega_1$.
This clearly implies that $(\nn^\nn, \precsim)$
contains a strictly increasing sequence of length
$\omega_1$, i.e. part (b) of the lemma is valid.
\medskip

\noindent \textit{Case 2.} $[Q]^2\subseteq \mathcal{D}$.
Let $(\tau_n)_n$ be a sequence in $Q$ which is
decreasing with respect to $\lex$. Arguing as in
Case 1 above, we see that the sequence $\big(h(\tau_n)\big)_n$
is strictly increasing. So, this case also implies
part (b) of the lemma.
\medskip

\noindent \textit{Case 3.} $[Q]^2\cap (\mathcal{I}\cup
\mathcal{D})=\varnothing$. We set $R=h(Q)$. Clearly
$R$ is a perfect subset of $\nn^\nn$. It is easy to check
that if $\alpha, \beta\in R$ with $\alpha\neq \beta$,
then $\alpha$ and $\beta$ are incomparable with respect
to $\precsim$. Hence $R$ is a perfect antichain of
$(\nn^\nn,\precsim)$, i.e. $(\nn^\nn, \precsim)$ is
not thin. Thus, this case implies part (a) of the lemma.
The proof is completed.
\end{proof}
\begin{lem}
\label{l3} Let $X$ be a separable Banach space such
that $\spw(X)$ is uncountable. Then there exists a
Cantor set $P\subseteq A_X$ consisting of pairwise
incomparable elements with respect to domination.
\end{lem}
\begin{proof}
Assume, towards a contradiction, that such a Cantor
set $P$ does not exist. This easily implies that
$(\nn^\nn,\precsim)$ is a thin quasi-order.
By Lemma \ref{l2}, we see that $(\nn^\nn,\precsim)$
is an $F_\sigma$ thin order that contains an
$\omega_1$-chain. But this possibility is ruled
out by Lemma \ref{ln1}. Having arrived to the desired
contradiction, the lemma is proved.
\end{proof}
\begin{rem}
We notice that Lemma \ref{l1} and Lemma \ref{l3} immediately
yield that if $X$ is a separable Banach space such that
$\spw(X)$ is uncountable, then $\spw(X)$ must contain an
antichain of the size of the continuum.
\end{rem}
We are ready to proceed to the proof of Theorem \ref{t1}.
\begin{proof}[Proof of Theorem \ref{t1}]
Let $P\subseteq A_X$ be the Cantor set obtained by Lemma
\ref{l3}. By passing to a perfect subset of $P$
if necessary, we may assume that
\begin{enumerate}
\item[(A)] for every $L\in P$ the sequence $(u_n)_{n\in L}$
is not equivalent to the standard basis of $\ell_1$.
\end{enumerate}
We will construct the family $(x_t)_{t\in\ct}$ by
``pulling back" inside $X$ the spreading models coded by $P$.
To this end, let $(d_m)_m$ be a countable
dense subset of $X$. Let $\mathrm{SPC}$ be the Borel
subset of $X^\nn$ defined in the proof of Lemma \ref{l1}.
Consider the following subset $G$ of $P\times [\nn]$ defined by
\begin{eqnarray*}
(L,M)\in G &\Leftrightarrow& \text{if } L=\{l_0<l_1<...\}
\text{ and } M=\{m_0<m_1<...\}, \text{ then}\\
& & \Big[ L\in P \text{ and } (d_{m_n})_n\in \mathrm{SPC}
\text{ and } \big(\exists \theta\geq 1 \\
& & \forall k \ \forall k\leq n_0<...< n_{k} \text{ we have }
(d_{m_{n_i}})_{i=0}^{k} \stackrel{\theta}{\sim}
(u_{l_i})_{i=0}^{k} \big)\Big].
\end{eqnarray*}
Let us gather some of the properties of the set $G$.
\begin{enumerate}
\item[(P1)] The set $G$ is Borel.
\item[(P2)] For every $(L,M)\in G$ and every $N$ infinite
subset of $M$, if $(y_n)_n$ is a spreading model generated
by a subsequence of $(d_m)_{m\in N}$, then $(y_n)_n$ is
equivalent to $(u_n)_{n\in L}$.
\item[(P3)] For every $L\in P$ there exists $M\in [\nn]$ such that
$(L,M)\in G$.
\item[(P4)] For every $(L,M)\in G$, the sequence $(d_m)_{m\in M}$
is weakly-null.
\end{enumerate}
Properties (P1) and (P2) are rather straightforward consequences
of the definition of the set $G$. Property (P3) follows by assumption
(A) above, the fact that $P$ is a subset of $A_X$ and a standard
perturbation argument. Property (P4) has already been verified in
the proof of Lemma \ref{l1}.

As $G$ is a Borel subset of $P\times [\nn]$, by (P3) above and
the Yankov-Von Neumann Uniformization Theorem (see \cite[Theorem 18.1]{Kechris}),
there exists a map $f:P\to [\nn]$ which is measurable
with respect to the $\sg$-algebra generated by the analytic sets
and such that $\big(L,f(L)\big)\in G$ for every $L\in P$. Notice
that the map $f$ must be one-to-one. Invoking the classical fact
that analytic sets have the Baire property, by \cite[Theorem 8.38]{Kechris}
and by passing to a perfect subset of $P$,
we may assume that $f$ is actually continuous. Moreover,
by passing to a further perfect subset of $P$ if necessary,
we may also assume that there exist $j_0, k_0\in \nn$ such
that for every $L\in P$, the sequence $(d_m)_{m\in f(L)}$ is
$j_0$-Schreier spreading and satisfies $\frac{1}{k_0}\leq\|d_m\|\leq k_0$
for every $m\in f(L)$.

The function $f$ is one-to-one and continuous.
Hence, identifying every element of $[\nn]$ with its characteristic
function (i.e. an element of $2^\nn$), we see that the set
$f(P)$ is a perfect subset of $2^\nn$. Recall that by
$\varphi:\ct\to\nn$ we denote the canonical bijection
described in the introduction. By recursion on the length of
finite sequences in $\ct$, we may easily select a family
$(m_s)_{s\in \ct}$ in $\nn$ with the following properties.
\begin{enumerate}
\item[(P5)] For every $s_1, s_2\in \ct$ we have $\varphi(s_1)<
\varphi(s_2)$ if and only if $m_{s_1}< m_{s_2}$.
\item[(P6)] For every $\sg\in 2^\nn$, setting $M_\sg=\{ m_{\sg|n}:
n\in\nn\}\in [\nn]$, there exist a unique $L_\sg\in P$ such that
$M_\sg\subseteq f(L_\sg)$.
\end{enumerate}
We set $x_s=d_{m_s}$ for every $s\in \ct$. We observe that
$\frac{1}{k_0}\leq \|x_s\|\leq k_0$ for all $s\in\ct$. We also
notice that for every $\sg\in 2^\nn$, the sequence $(x_{\sg|n})_n$
is $j_0$-Schreier spreading.

Now let $s\in\ct$ with $|s|=k$ and $\sg\in 2^\nn$ with $\sg|k=s$.
By properties (P4) and (P6), we see that the sequence
$(x_{\sg|n})_{n>k}$ is weakly-null. Using this observation
and the classical procedure of Mazur for constructing Schauder
basic sequences (see \cite{LT}), we may select a family
$(s_t)_{t\in\ct}$ in $\ct$ such that, setting $x_t=x_{s_t}$
for every $t\in\ct$, the following are satisfied.
\begin{enumerate}
\item[(P7)] For every $t_1, t_2\in \ct$ we have that
$s_{t_1}\sqsubset s_{t_2}$ if and only if $t_1\sqsubset t_2$.
Moreover, $|s_{t_1}|<|s_{t_2}|$ if and only if $\varphi(s_1)<
\varphi(s_2)$.
\item[(P8)] If $(t_n)_n$ is the enumeration of $\ct$
according to $\varphi$, then the sequence $(x_{t_n})_n$
is Schauder basic.
\end{enumerate}
It is easy to verify that the family $(x_t)_{t\in\ct}$ has
all properties stated in Theorem \ref{t1}. The proof is completed.
\end{proof}
\begin{rem}
We would like to remark few things on the richness of
the structure $(\spw(X),\leq)$ when $\spw(X)$ is uncountable.
Let $X$ be a separable Banach space and assume that there
exist $C\geq 1$ and a family $\{(y^\xi_n)_n:\xi<\omega_1\}$ of
mutually inequivalent spreading models generated by weakly-null
sequences in $X$ such that for every $\xi<\zeta<\omega_1$ either
the sequence $(y^\xi_n)_n$ is $C$-dominated by $(y^\zeta_n)_n$
or vice versa. By Lemma \ref{l1}, there exist $K\geq 1$ and
$U\subseteq A_X$ uncountable such that the following hold.
For every $L, M\in U$ either $(u_n)_{n\in L}$ is $K$-dominated
by $(u_n)_{n\in M}$ or vice versa, and moreover, for
every $L\in U$ there exists a unique ordinal $\xi_L<\omega_1$
such that $(u_n)_{n\in L}$ is equivalent to $(y^{\xi_L}_n)_n$.
Let $\overline{U}$ be the closure of $U$
in $[\nn]$ and set $F=\overline{U}\cap A_X$. Then $F$
is an uncountable analytic set. Consider the following
symmetric relation $\thickapprox_{K}$ in $[\nn]\times[\nn]$
defined by
\[ L\thickapprox_K M\Leftrightarrow \text{either } (u_n)_{n\in L}
\text{ is } K\text{-dominated by } (u_n)_{n\in M}
\text{ or vice versa}.\]
It is easy to see that $\thickapprox_K$ is closed in
$[\nn]\times[\nn]$. By the choice of $U$, we have
$L\thickapprox_K M$ for every $L,M\in U$. As $\thickapprox_K$
is closed, we see that $L\thickapprox_K M$ for every
$L,M\in \overline{U}$. In particular, $L\thickapprox_K M$
for every $L,M\in F$. Notice that $U\subseteq F$, and so,
the relation $\sim$ of equivalence restricted on $F$ has
uncountable many equivalence classes. By Lemma \ref{ln2},
there exists a perfect subset $P$ of $F$ such that
for every $L,M\in P$ the sequences $(u_n)_{n\in L}$
and $(u_n)_{n\in M}$ are not
equivalent\footnote[2]{This does not follow directly by Lemma
\ref{ln2} as $F$ is not Polish. One has to observe that
$F$ is the continuous surjective image of $\nn^\nn$ and use
an argument as in the beginning of Section 4.}.
Thus, we have shown the following.
\begin{prop}
\label{newp1} Let $X$ be a separable Banach space and assume
that there exist $C\geq 1$ and a family $\{(y^\xi_n)_n:\xi<\omega_1\}$
of mutually inequivalent spreading models generated by weakly-null
sequences in $X$ such that for every $\xi<\zeta<\omega_1$ either
the sequence $(y^\xi_n)_n$ is $C$-dominated by $(y^\zeta_n)_n$
or vice versa. Then $(\spw(X),\leq)$ contains a linearly ordered
subset of the size of the continuum.
\end{prop}
Related to Proposition \ref{newp1}, the following question
is open to us. Let $X$ be a separable Banach space and assume
that $\spw(X)$ is uncountable. Does $(\spw(X),\leq)$ contain
a linearly ordered subset of the size of the continuum,
or at least uncountable?
\end{rem}


\section{Proof of Theorem \ref{t2}}

\noindent (a) First we need to recall some standard facts
(see \cite{Kechris}, page 351). Let $S$ be
a set and $\prec$ a strict, well-founded (binary) relation on $S$.
This is equivalent to asserting that there is no infinite
decreasing chain $\cdots \prec s_1\prec s_0$. By recursion
on $\prec$, we define the \textit{rank} function $\rho_\prec:
S\to \mathrm{Ord}$ of $\prec$ by the rule
\[ \rho_\prec(s)=\sup\{ \rho_\prec(x)+1: x\prec s\}.\]
In particular, $\rho_\prec(s)=0$ if and only if $s$
is minimal. The \textit{rank} $\rho(\prec)$ of $\prec$
is defined by $\rho(\prec)=\sup\{ \rho_\prec(s)+1: s\in S\}$.

We are ready to proceed to the proof.
So, let $X$ be a separable Banach space such that
$\spw(X)$ contains a strictly decreasing sequence of
length $\omega_1$. Let $A_X$ be the analytic subset
of $[\nn]$ obtained by Lemma \ref{l1}. Consider
the following relation $\prec$ on $[\nn]$ defined by
\[ L\prec M \Leftrightarrow (L\in A_X) \text{ and }
(M\in A_X) \text{ and } (M<L). \]
That is, $\prec$ is the relation $>$ (the reverse
of $<$) restricted on $A_X\times A_X$. Clearly $\prec$
is analytic (as a subset of $[\nn]\times [\nn]$). Let
$\{(y_n^\xi)_n:\xi<\omega_1\}$ be a strictly decreasing sequence
in $\spw(X)$. By Lemma \ref{l1}, for every $\xi<\omega_1$
we may select $L_\xi\in A_X$ such that $(u_n)_{n\in L_{\xi}}$ is
equivalent to $(y^\xi_n)_n$. It follows that
$L_\xi<L_\zeta$ if and only if $\zeta<\xi$.

Assume, towards a contradiction,
that $\spw(X)$ does not contain a strictly increasing
sequence of length $\omega_1$. Then, by the result of
Sari \cite{S} already quoted in the proof of Theorem \ref{t1},
$\spw(X)$ does not contain a strictly increasing sequence
of length $\omega$. It follows that $\prec$ is a well-founded
relation on $[\nn]$ which is in addition analytic. By
the Kunen-Martin Theorem (see \cite[Theorem 31.5]{Kechris}),
we see that $\rho(\prec)$ is a countable ordinal, say $\xi_0$.
For every $\eta< \xi_0$ let
\[ A_X^\eta=\{ L\in A_X: \rho_\prec(L)=\eta\}. \]
As $\rho_\prec(L)<\xi_0$ for every $L\in A_X$
we see that $A_X= \bigcup_{\eta<\xi_0} A_X^\eta$.
Moreover, for every $L, M\in A_X^\eta$ we have that
either $L\sim M$ or $L\perp M$. That is, we have
partitioned the quotient $A_X/\sim$ into countable many antichains.
As the family $\{L_\xi:\xi<\omega_1\}$ is uncountable, we see that
there exist $\xi,\zeta<\omega_1$ with $\xi\neq \zeta$
and $\eta< \xi_0$ such that $L_\xi, L_\zeta\in A_X^{\eta}$.
But this is clearly impossible. Having arrived to the
desired contradiction the proof of part (a)
is completed.\\
(b) Again we need to discuss some standard facts.
Let $R$ be a binary relation on $\nn$, i.e. $R\subseteq \nn\times\nn$.
By identifying $R$ with its characteristic function, we
view every binary relation on $\nn$ as an element of $2^{\nn\times\nn}$.
Let $\mathrm{LO}$ be the subset of $2^{\nn\times\nn}$
consisting of all (strict) linear orderings on $\nn$.
It is easy to see that $\mathrm{LO}$ is a closed
subset of $2^{\nn\times\nn}$ (see also \cite{Kechris},
page 212). For every $\alpha\in\mathrm{LO}$ and every $n,m\in\nn$
we write
\[ n<_\alpha m \Leftrightarrow \alpha(n,m)=1. \]
Let $\mathrm{WO}$ be the subset of $\mathrm{LO}$ consisting of
all well-orderings on $\nn$. For every $\alpha\in \mathrm{WO}$,
$|\alpha|$ stands for the unique ordinal which is isomorphic to
$(\nn,<_\alpha)$. We will need the following Boundedness
Principle for $\mathrm{WO}$ (see \cite{Kechris}, page 240):
if $B$ is an analytic subset of $\mathrm{WO}$, then
$\sup\{|\alpha|:\alpha\in B\}<\omega_1$.

We proceed to the proof of part (b). Let $X$ be a separable
Banach space. Let $A_X$ be the analytic subset of $[\nn]$
obtained by Lemma \ref{l1}. Consider the following
subset $\mathrm{O}_X$ of $\mathrm{LO}$ defined by
\begin{eqnarray*}
\alpha\in \mathrm{O}_X & \Leftrightarrow &
\exists (L_n)_n\in \big([\nn]\big)^\nn \text{ with }
\Big[ (\forall n \ L_n\in A_X) \text{ and}\\
& & \big[ \forall n,m \ (n<_\alpha m \Leftrightarrow L_n>L_m) \big]\Big].
\end{eqnarray*}
As $A_X$ is analytic, it easy to check that $\mathrm{O}_X$ is
an analytic subset of $\mathrm{LO}$.
\medskip

\noindent \textsc{Claim 2.} \textit{The set
$\spw(X)$ does not contain a strictly increasing
sequence if and only if $\mathrm{O}_X\subseteq \mathrm{WO}$.}
\medskip

\noindent \textit{Proof of Claim 2.} First assume that
there exists $\alpha\in \mathrm{O}_X$ with $\alpha\notin \mathrm{WO}$.
By definition, there exists a sequence $(L_n)_n$ in $A_X$
such that for all $n,m\in \nn$ we have
\[ n<_\alpha m\Leftrightarrow L_n>L_m. \]
As $\alpha\notin\mathrm{WO}$, there exists a sequence
$(n_i)_i$ in $\nn$ such that $n_{i+1}<_\alpha n_i$ for
all $i\in\nn$. It follows that $(L_{n_i})_i$
is a strictly increasing sequence, which clearly
implies that $\spw(X)$ contains a strictly increasing sequence.

Conversely, assume that $\spw(X)$ contains a strictly
increasing sequence. Hence, we may find a sequence
$(L_n)_n$ in $A_X$ such that $L_n<L_m$ if and only if $n<m$.
Let $\alpha\in\mathrm{LO}$ be defined by
\[n<_\alpha m \Leftrightarrow n>m \ (\Leftrightarrow L_n>L_m).\]
Then $\alpha\in \mathrm{O}_X$ and $\alpha\notin \mathrm{WO}$.
The claim is proved.   \hfill $\lozenge$
\medskip

\noindent Now, let $X$ be a separable Banach space that does not
contain a strictly increasing sequence. By Claim 2,
we see that the set $\mathrm{O}_X$ is an analytic subset
of $\mathrm{WO}$. Hence, by boundedness, we see that
\[ \sup\{ |\alpha|:\alpha\in \mathrm{O}_X\}=\xi_X<\omega_1.\]
We claim that $\xi_X$ is the desired ordinal. Indeed, let
$\xi$ be a countable ordinal and $\{ (y^\zeta_n)_n:
\zeta<\xi\}$ be a strictly decreasing sequence in $\spw(X)$.
By Lemma \ref{l1}, we may find $(L_\zeta)_{\zeta<\xi}$ in
$A_X$ which is strictly decreasing. Fix a bijection
$e:\nn\to \{\zeta: \zeta<\xi\}$ and define
$\alpha\in\mathrm{WO}$ by
\[ n<_\alpha m \Leftrightarrow e(n)<e(m) \ (\Leftrightarrow
L_{e(n)}>L_{e(m)}).\]
It follows that $\alpha\in \mathrm{O}_X$, and so,
$\xi=|\alpha|\leq\xi_X$. The proof is completed.
\begin{rem}
Denote by $\mathrm{SB}$ the standard Borel space of all
separable Banach spaces as it is discussed in \cite{AD}, \cite{B}
and \cite{Kechris}. Consider the subset $\mathrm{NCI}$
of $\mathrm{SB}$ defined by
\[ X\in \mathrm{NCI} \Leftrightarrow \spw(X) \text{ does
not contain a strictly increasing infinite sequence}.\]
It can be shown, using some results from \cite{DOS}, that the
set $\mathrm{NCI}$ is co-analytic non-Borel in $\mathrm{SB}$.
Moreover, there exists a co-analytic rank
$\phi:\mathrm{NCI}\to\omega_1$ on $\mathrm{NCI}$ such that
for every $X\in\mathrm{NCI}$ we have
\[ \sup\{ |\alpha|:\alpha\in\mathrm{O}_X\}\leq \phi(X)\]
where $\mathrm{O}_X$ is as in the proof of Theorem \ref{t2}(b)
(for the definition of co-analytic ranks we refer to \cite{Kechris}
while for applications of rank theory to Banach space Theory
we refer to \cite{AD}).
\end{rem}



\begin{thebibliography}{99}

\bibitem[AOST]{AOST} G. Androulakis, E. Odell, Th. Schlumprecht and
N. Tomczak-Jaegermann, \textit{On the structure of the spreading models
of a Banach space}, Canadian J. Math., 57 (2005), 673-707.

\bibitem[AD]{AD} S. A. Argyros and P. Dodos, \textit{Genericity
and amalgamation of classes of Banach spaces}, Adv. Math.,
209 (2007), 666-748.

\bibitem[AT]{AT} S. A. Argyros and S. Todor\v{c}evi\'{c}, \textit{Ramsey
Methods in Analysis}, Advanced Courses in Mathematics, CRM Barcelona,
Birkh\"{a}user, Verlag, Basel, 2005.

\bibitem[B]{B} B. Bossard, \textit{A coding of separable Banach
spaces. Analytic and coanalytic families of Banach spaces},
Fund. Math., 172 (2002), 117-152.

\bibitem[DOS]{DOS} S. J. Dilworth, E. Odell and B. Sari, \textit{Lattice
structures and spreading models}, Israel J. Math. (to appear).

\bibitem[FR]{FR} V. Ferenczi and C. Rosendal, \textit{Complexity
and homogeneity in Banach spaces} (preprint).

\bibitem[F]{F} H. Friedman, \textit{Borel structures in mathematics},
manuscript, Ohio State University, 1979.

\bibitem[HMS]{HMS} L. Harrington, D. Marker and S. Shelah, \textit{Borel
orderings}, Trans. AMS, 310 (1988), 293-302.

\bibitem[HS]{HS} L. Harrington and S. Shelah, \textit{Counting the equivalence
classes of co-$\kappa$-Suslin equivalence relations}, Logic Colloquium '80,
North-Holland, Amsterdam, 1982, 147-152.

\bibitem[Ke]{Kechris} A. S. Kechris, \textit{Classical Descriptive Set Theory},
Grad. Texts in Math., 156, Springer-Verlag, 1995.

\bibitem[LT]{LT} J. Lindenstrauss and L. Tzafriri, \textit{Classical Banach spaces
I and II}, Springer, 1996.

\bibitem[L]{L} A. Louveau, \textit{Two results on Borel orders}, Journal
Symb. Logic, 54 (1989), 865-874.

\bibitem[LStR]{LStR} A. Louveau and J. Saint Raymond, \textit{On the
quasi-ordering of Borel linear orders under embeddability}, Journal
Symb. Logic, 55 (1990), 537-560.

\bibitem[MK]{MK} D. A. Martin and A. S. Kechris, \textit{Infinite
games and effective descriptive set theory}, in \textit{Analyitc
Sets}, editors C. A. Rogers et al., Academic Press, 1980.

\bibitem[P]{P} A. Pelczynski, \textit{Universal bases}, Studia
Math., 19 (1960), 247-268.

\bibitem[Ros]{Ros} C. Rosendal, \textit{Incomparable, non-isomorphic and
minimal Banach spaces}, Fund. Math., 183 (2004), 253-274.

\bibitem[Sa]{S} B. Sari, \textit{On Banach spaces with few spreading models},
Proc. AMS, 134 (2005), 1339-1345.

\bibitem[Si]{Si} J. H. Silver, \textit{Counting the number of equivalence
classes of Borel and coanalytic equivalence relations}, Ann. Math.
Logic, 18 (1980), 1-28.

\end{thebibliography}
\end{document}